\documentclass[letterpaper, 10 pt, conference]{ieeeconf}

\IEEEoverridecommandlockouts    
\usepackage{psfrag}
\usepackage{amsmath,amssymb,amsfonts,bbm}
\usepackage{latexsym}
\usepackage{graphicx}

\usepackage{colortbl}
\usepackage{fancyhdr}
\usepackage{epstopdf}
\usepackage{float}
\usepackage{hyperref}
\usepackage{booktabs}
\usepackage{enumerate}
\usepackage[ruled]{algorithm2e}
\usepackage{balance}
\usepackage{mathrsfs}

\usepackage{tikz}
\usepackage{pgfplots}
\pgfplotsset{compat=1.10}
\usepgfplotslibrary{fillbetween}

\usepackage[font=small , skip = -6pt]{caption}

\newtheorem{theorem}{Theorem}
\newtheorem{definition}{Definition}
\newtheorem{proposition}{Proposition}
\newtheorem{lemma}{Lemma}

\newtheorem{remark}{Remark}

\newtheorem{standing}{Standing Assumption}

\IEEEoverridecommandlockouts

\newcommand{\R}{\mathbb{R}}

\newcommand{\Ncal}{\mathcal{N}}
\newcommand{\bOmega}{\boldsymbol{\Omega}}
\newcommand{\bomega}{\boldsymbol{\omega}}
\newcommand{\Xcal}{\mathcal{X}}
\newcommand{\bXcal}{\boldsymbol{\mathcal{X}}}

\newcommand{\mc}{\mathcal}

\newcommand{\boldx}{ \boldsymbol{x} }
\newcommand{\xMi}{ \boldsymbol{x}_{-i} }
\newcommand{\boldy}{ \boldsymbol{y} }

\newcommand{\boldz}{ \boldsymbol{z} }
\newcommand{\boldv}{ \boldsymbol{v} }
\newcommand{\bolds}{ \boldsymbol{\sigma} }

\newcommand{\col}{\mathrm{col}}
\newcommand{\nc}{\mathrm{N}}

\newcommand{\fix}{\mathrm{fix}}
\newcommand{\zer}{\mathrm{zer}}
\newcommand{\Id}{\mathrm{Id}}

\newcommand{\argmin}{\mathrm{argmin}}
\newcommand{\dom}{\mathrm{dom}}
\newcommand{\proj}{\mathrm{proj}}
\newcommand{\diag}{\mathrm{diag}}

\newcommand{\ellF}{\ell_{\text{F}}}

\newcommand{\bs}{\boldsymbol}
\newcommand{\norm}[1]{\left\|#1\right\|}

\graphicspath{{Figures/} }

\newcommand{\Rmnum}[1]{\expandafter\@slowromancap\romannumeral #1@}

\begin{document}

\title{Projected-gradient algorithms for Generalized Equilibrium seeking in Aggregative Games are preconditioned Forward-Backward methods}

\author{Giuseppe Belgioioso \and Sergio Grammatico 
\thanks{G. Belgioioso is with the Control Systems group, TU Eindhoven, The Netherlands. S. Grammatico is with the Delft Center for Systems and Control (DCSC), TU Delft, The Netherlands. 
E-mail addresses: \texttt{g.belgioioso@tue.nl}, \texttt{s.grammatico@tudelft.nl}. This work was partially supported by the Netherlands Organisation for Scientific Research (NWO) under research projects OMEGA (grant n. 613.001.702) and P2P-TALES (grant n. 647.003.003).\smallskip
 \newline
}
}

\thispagestyle{empty}
\pagestyle{empty}

\maketitle

\begin{abstract}
We show that projected-gradient methods for the distributed computation of generalized Nash equilibria in aggregative games are preconditioned forward-backward splitting methods applied to the KKT operator of the game. Specifically, we adopt the preconditioned forward-backward design, recently conceived by Yi and Pavel in the manuscript ``A distributed primal-dual algorithm for computation of generalized Nash equilibria via operator splitting methods'' for generalized Nash equilibrium seeking in aggregative games. 
Consequently, we notice that two projected-gradient methods recently proposed in the literature are preconditioned forward-backward methods. More generally, we provide a unifying operator-theoretic ground to design projected-gradient methods for generalized equilibrium seeking in aggregative games.
\end{abstract}

\section{Introduction}
Aggregative game theory \cite{kukushkin:04} is a mathematical framework to model the interdependent optimal decision making problems for a set of noncooperative agents,
whenever the decision of each agent is affected by some aggregate effect of all the agents. This feature emerges in several application areas, such as demand side management in the smart grid \cite{Saad2012}, e.g. for electric vehicles \cite{parise:colombino:grammatico:lygeros:14, ma:zou:ran:shi:hiskens:16} and thermostatically controlled loads \cite{grammatico:gentile:parise:lygeros:15, Li2016}, demand response in competitive markets \cite{li2015demand} and network congestion control \cite{barrera:garcia:15}.

Existence and uniqueness of Nash equilibria in (aggregative) noncooperative games is well established in the literature of operation research \cite{facchinei2010generalized}, \cite[\S12]{palomar2010convex}, and automatic control \cite{pavel:07, kulkarni:shanbhag:12}. For the computation of a game equilibrium, several algorithms are available, both distributed protocols \cite{salehisadaghiani:pavel:16,koshal:nedic:shanbhag:16} and semi-decentralized schemes \cite{grammatico:parise:colombino:lygeros:16,paccagnan:16, grammatico:17, belgioioso:grammatico:17cdc}. Among these, an elegant approach is to characterize the desired equilibrium solutions as the zeros of an operator, possibly monotone, e.g.\ the concatenation of interdependent Karush--Kuhn--Tucker operators, and in turn formulate an equivalent fixed-point problem, which is solved via appropriate fixed-point iterations. Overall, the available methods can ensure global convergence to an equilibrium if the coupling among the cost functions of the agents is ``well behaved'', e.g.\ if the problem data are convex and the so-called pseudo-gradient game mapping is (strictly, strongly) monotone \cite{belgioioso2017convexity}.

A popular class of algorithms for Nash equilibrium seeking is that of projected-gradient algorithms \cite[\S 12]{facchinei:pang}, \cite{koshal:nedic:shanbhag:16,paccagnan:16,belgioioso2017convexity}. 
Whenever the pseudo-gradient game mapping is strongly monotone, projected-gradient algorithms can ensure fast convergence to a Nash equilibrium, possibly via distributed computation and information exchange. It follows that projected-gradient methods have the potential to be fast, simple and scalable with respect to the population size. At the same time, in the context of Nash equilibrium seeking, the convergence analyses for the available projected-gradient methods are quite diverse in nature.

In this paper, we aim at a unifying convergence analysis for projected-gradient algorithms that are adopted for the computation of generalized Nash equilibria in aggregative games. Specifically, we adopt a general perspective based on monotone operator theory \cite{bauschke2011convex} to show that projected-gradient algorithms with sequential updates belong to the class of preconditioned forward-backward splitting methods, introduced in \cite{yi2017distributed} for multi-agent network games.

The main technical contribution of the paper is to conceive a design procedure for the preconditioned forward-backward splitting method. The proposed design is based not only on the splitting of the monotone operator whose zeros are the game equilibria, but also on the choice of the so-called preconditioning matrix, which induces the quadratic norm adopted to show global convergence of the resulting algorithm. Since the convergence characterization of the forward-backward splitting method is well established, the advantage of the proposed design is that global convergence follows provided that some mild monotonicity assumptions on the problem data are satisfied.

Remarkably, we discover that two recent projected-gradient algorithms for Nash equilibrium seeking in aggregative games, \cite{paccagnan:16} and \cite{koshal:nedic:shanbhag:16}, can be equivalently written as preconditioned forward-backward splitting methods with symmetric preconditioning matrix, despite their algorithmic formulation is ``asymmetric''. 

%The paper is organized as follows: Section \ref{sec:GAG} introduces the game theoretic setup and characterizes the generalized Nash equilibrium as equilibrium solution. 
%In Section \ref{sec:MonRef} , we adopt the operator theoretic perspective to the generalized Nash equilibrium problem. In Section \ref{sec.IFBS}, we introduce the preconditioned forward-backward splitting method and show that it encompasses two projected-gradient methods as special cases. In Section \ref{sec:CA}, we study the convergence of the preconditioned forward-backward splitting method. Section \ref{sec:Concl} summarizes the message of the paper.

\subsection*{Basic notation}
$\R$ denotes the set of real numbers, and $\overline{\R} := \R \cup \{\infty\}$ the set of extended real numbers. $\bs{0}$ ($\bs{1}$) denotes a matrix/vector with all elements equal to $0$ ($1$); to improve clarity, we may add the dimension of these matrices/vectors as subscript. $A \otimes B$ denotes the Kronecker product between matrices $A$ and $B$; $\left\| A \right\|$ denotes the maximum singular value of $A$; $\rm{eig}(A)$ denotes the set of eigenvalues of $A$.
Given $N$ vectors $x_1, \ldots, x_N \in \R^n$, $\boldsymbol{x} := \col\left(x_1,\ldots,x_N\right) = \left[ x_1^\top, \ldots , x_N^\top \right]^\top$.

\subsection*{Operator theoretic definitions}
$\Id(\cdot)$ denotes the identity operator. The mapping $\iota_{S}:\R^n \rightarrow \{ 0, \, \infty \}$ denotes the indicator function for the set $\mc{S} \subseteq \R^n$, i.e., $\iota_{S}(x) = 0$ if $x \in S$, $\infty$ otherwise. For a closed set $S \subseteq \R^n$, the mapping $\proj_{S}:\R^n \rightarrow S$ denotes the projection onto $S$, i.e., $\proj_{S}(x) = \argmin_{y \in S} \left\| y - x\right\|$. The set-valued mapping $\nc_{S}: \R^n \rightrightarrows \R^n$ denotes the normal cone operator for the the set $S \subseteq \R^n$, i.e., 
$\nc_{S}(x) = \varnothing$ if $x \notin S$, $\left\{ v \in \R^n \mid \sup_{z \in S} \, v^\top (z-x) \leq 0  \right\}$ otherwise.
For a function $\psi: \R^n \rightarrow \overline{\R}$, $\dom(\psi) := \{x \in \R^n \mid \psi(x) < \infty\}$; $\partial \psi: \dom(\psi) \rightrightarrows {\R}^n$ denotes its subdifferential set-valued mapping, defined as $\partial \psi(x) := \{ v \in \R^n \mid \psi(z) \geq \psi(x) + v^\top (z-x)  \textup{ for all } z \in {\rm dom}(\psi) \}$;
A set-valued mapping $\mathcal{F} : \R^n \rightrightarrows \R^n$ is $\ell$-Lipschitz continuous, with $\ell>0$, if $\|u-v\| \leq \ell\|x-y\|$ for all $x,y \in \R^n$, $u \in \mathcal{F} (x)$, $v \in \mathcal{F} (y)$; $\mathcal{F} $ is (strictly) monotone if $(u-v)^\top  (x-y) \geq (>) \, 0$ for all $x \neq y \in \R^n$, $u \in \mathcal{F} (x)$, $v \in \mathcal{F} (y)$; $\mathcal{F} $ is $\eta$-strongly monotone, with $\eta>0$, if 
$(u-v)^\top (x-y) \geq \eta \left\| x-y \right\|^2$ for all $x,y \in \R^n$, $u \in \mathcal{F} (x)$, $v \in \mathcal{F} (y)$; $\mathcal{F} $ is $\eta$-{averaged}, with $\eta \in (0,1)$, if  
$\left\| \mathcal{F} (x) - \mathcal{F} (y) \right\|^2 \leq \left\| x-y \right\|^2 - \tfrac{1-\eta}{\eta}\left\| \left( \textup{Id}-\mathcal{F}  \right)(x) - \left(\textup{Id}-\mathcal{F}  \right)(y) \right\|^2$, for all $x, y \in \R^n$;
%$F$ is nonexpansive if $\left\| F(x) - F(y) \right\| \leq \left\| x-y \right\|^2$, for all $x, y \in \R^n$;
$\mathcal{F} $ is $\beta$-cocoercive, with $\beta>0$, if $\beta \mathcal{F} $ is $\tfrac{1}{2}$-averaged.
With ${\rm J}_{\mathcal{F} }:=(\Id + \mathcal{F} )^{-1}$, we denote the resolvent operator of $\mathcal{F} $, which is $\tfrac{1}{2}$-averaged if and only if $\mathcal{F} $ is monotone; $\fix\left( \mathcal{F}\right) := \left\{ x \in \R^n \mid x \in \mathcal{F}(x) \right\}$ and $\zer\left( \mathcal{F}\right) := \left\{ x \in \R^n \mid 0 \in \mathcal{A}(x) \right\}$ denote the set of fixed points and of zeros, respectively.

\section{Generalized aggregative games} \label{sec:GAG}

\subsection{Mathematical formulation}
We consider a set of $N$ noncooperative agents, where each agent $i \in \Ncal := \{1,\cdots, N \}$ shall choose its decision variable (i.e., strategy) $x_i$ from the local decision set $\Omega_i \subseteq \mathbb{R}^n$ with the aim of minimizing its local cost function $\left( x_i, \bs{x}_{-i} \right) \mapsto J_i\left( x_i, \bs{x}_{-i} \right): \R^n \times \R^{n(N-1)} \rightarrow \overline{\R}$, which depends on both the local variable $x_i$ (first argument) and on the decision variables of the other agents, $\bs{x}_{-i} = \col\left( \{ x_j \}_{j \neq i} \right)$ (second argument).
%We emphasize that while each function $J_i$ depends on the overall decision variable $\boldx = \col\left( x_1, \ldots, x_N\right)$, each agent $i$ has decision authority on its variable $x_i$ only.

We focus on the class of aggregative games, where the cost function of each agent depends on the local decision variable and on the value of the aggregation function $\sigma: \boldsymbol{\Omega} \rightarrow \frac{1}{N} \sum_{j=1}^N \Omega_j \subseteq \mathbb{R}^n$, with $\boldsymbol{\Omega} := \Omega_1 \times \ldots \times \Omega_N$. In particular, we consider average aggregative games, where the aggregation function is the average function, i.e., 
\begin{equation} \label{eq:sigma} \textstyle
\sigma(\boldsymbol{x}) := M \boldx = \frac{1}{N}\sum_{i=1}^{N} x_i, \text{ hence } M := \frac{1}{N}\mathbf{1}^\top_N \otimes I_n.
\end{equation}

Thus, for each $i \in \Ncal$, there is a function $f_i : \R^n \times \R^n \rightarrow \overline{\R}$ such that the local cost function $J^i$ can be written as
\begin{equation} \label{eq:CFi} \textstyle
J_i(x_i, \xMi) =: \textstyle  f_i \left( x_i, \sigma(\boldsymbol{x}) \right). 
\end{equation}

Furthermore, we consider \textit{generalized games}, where the coupling among the agents arises not only via the cost functions, but also via their feasible decision sets. In our setup, the coupling constraints are described by an affine function, $x \mapsto A \boldx - b$, where $A \in \R^{m \times nN}$ and $b \in \R^m$. Thus, the collective feasible set, $\bXcal \subseteq \R^{nN}$, reads as
\begin{equation}
\label{eq:G}
%%\boldsymbol{x} \in \boldsymbol{\mathcal{C}} := \{ \boldsymbol{y} \in \mathbb{R}^{nN} \mid \, g( \boldsymbol{y} ) \leq \boldsymbol{0}_m \},
\bXcal \:= \boldsymbol{\Omega}  \cap \left\{ \boldy \in \R^{nN} | \, A \boldy - b \leq \boldsymbol{0}_m \right\};
\end{equation}

while the feasible decision set of each agent $i \in \Ncal$ is characterized by the set-valued mapping $\Xcal_i$,
%:\R^{n(N-1)}\rightrightarrows \R^n$
defined as
$
\Xcal_i(\xMi) := \big\{ y_i \in \Omega_i | \, A_i y_i \leq b -\sum_{j \neq i}^N A_j x_j \big\},
$
where $A_i \in \R^{m \times n}$ and $A = \left[ A_1, \ldots, A_N \right]$.
The set $\Omega_i$ represents the local decision set for agent $i$, while the matrix $A_i$ defines how agent $i$ is involved in the coupling constraints. For instance, the shared constraints in \eqref{eq:G} may contain a sparsity pattern that can be defined via a graph, where each agent has a set of ``neighbors'' with whom to share some constraints.
\begin{remark}[Affine coupling constraint]
Affine coupling constraints as considered in this paper are very common in the literature of noncooperative games, see \cite{paccagnan:16},
\cite{grammatico:17}, \cite{yi2017distributed}, \cite{liang2017distributed}.
Moreover, we recall that the more general case with separable convex coupling constraints can be reformulated as game via affine coupling constraints \cite[Remark 2]{grammatico:17tcns}.
{\hfill $\square$}
\end{remark}

\smallskip
Next, let us postulate standard convexity and compactness assumptions for the constraint sets, convexity and differentiability assumptions for the local cost functions.

\smallskip
\begin{standing}[Convex differentiable functions] \label{ass:CCF}
For each $i\in \Ncal$ and $\boldsymbol{y} \in \boldsymbol{\mathcal{X}}_{-i}$
the function $J_i\left( \, \cdot \, , \, \boldsymbol{y} \right)$ is convex and continuously differentiable.
{\hfill $\square$}
\end{standing} 

\smallskip
\begin{standing}[Compact convex constraints] \label{ass:CS}
For  each $i \in \Ncal$, the set $\Omega_i$ is nonempty, compact and convex. The set $\bXcal$ satisfies the Slater's constraint qualification.
{\hfill $\square$}
\end{standing}
\smallskip

In summary, the aim of each agent $i$, given the decision variables of the other agents, $\xMi$, is to choose a strategy, $x_i$, that solves its local optimization problem according to the game setup previously described, i.e., 
\begin{align}\label{eq:Game}
\left\{
\begin{matrix}
\underset{x_i \in \, \Omega_i}{\operatorname{min}}& \; %f_i \big( x_i, \frac{1}{N}x_i + \frac{1}{N} \sum_{j \neq i}^N x_j \big)   \hspace*{2em}  \\
J_i \big( x_i, \xMi \big)   \hspace*{3.8em}  \\
\text{ s.t. }   &   A_i x_i \leq b-\sum_{j \neq i}^N A_j x_j
\end{matrix} \qquad
\forall i\in \Ncal.
\right.
\end{align}
From the game-theoretic perspective, we consider the problem to compute a Nash equilibrium, as formalized next.

\smallskip
\begin{definition}[Generalized Nash equilibrium]
The collective strategy $\boldx^*$ is a generalized Nash equilibrium (GNE) of the game in \eqref{eq:Game} if $\boldx^* \in \bXcal$ and for all $i\in \Ncal$
\begin{align*}
% \label{eq:GNE}
J_i\left( x^*_{i}, \boldx^*_{-i} \right) \leq 
\inf\left\{ J_{i}(y, \, \boldx^*_{-i}) \, \mid \, y \in \Xcal_i(\boldx^*_{-i}) \right\}.
\end{align*}
{\vspace*{-2em \hfill $\square$}}
%\hfill $\square$
\end{definition}
\medskip

In other words, a set of strategies is a Nash equilibrium if no agent can improve its objective function by unilaterally changing its strategy to another feasible one.
%The problem to find such a set of strategies is known as generalized Nash equilibrium problem (GNEP). 

Under Assumptions \ref{ass:CCF}$-$\ref{ass:CS}, the existence of a GNE of the game in \eqref{eq:Game} follows from Brouwer's fixed-point theorem \cite[Proposition~12.7]{palomar2010convex}, while uniqueness does not hold in general.

\subsection{Variational equilibria and pseudo-gradient of the game} \label{subsec:VA}

Within all the possible Nash equilibria, we focus on an important subclass of equilibria, with some relevant structural properties, such as ``larger social stability'' and ``economic fairness'' \cite[Theorem 4.8]{facchinei2010generalized}, that corresponds to the solution set of an appropriate variational inequality. Let us first formalize the notion of variational inequality problem.

\smallskip
\begin{definition}[Generalized variational inequality]
Consider a closed convex set $S \subseteq \R^n$, 
a set-valued mapping $\Psi:S \rightrightarrows \R^n$, and a single-valued mapping $\psi: S \rightarrow \R^n$. 
The generalized variational inequality problem GVI$(S,\Psi)$, is the problem to find $x^* \in S$ and $g^* \in \Psi(x^*)$ such that
$$ \textstyle
(x-x^*)^\top \, g^* \geq 0 \ \textup{  for all } x\in S.
$$
If $\Psi (x) = \{\psi(x)\}$ for all $x\in S$, then GVI$(S,\Psi)$ reduces to the variational inequality problem VI$(S,\psi)$.
%which is the problem to find a vector $x^* \in S$ such that
%\begin{equation*}
%(x-x^*)^\top  \, \psi(x^*) \geq 0 \ \textup{  for all } x\in S.
%\end{equation*}
%{\vspace*{-2em \hfill $\square$}}
\end{definition}
\smallskip

A fundamental mapping in a noncooperative game is the so-called \textit{pseudo-gradient}, $F: \bXcal \rightrightarrows \R^{nN}$, defined as
\begin{align} 
\label{eq:PsGr}
F(\boldx) &:= 
\col\left( \left\{  \partial_{x_i} \, J_i \left( x_i, \, \boldsymbol{x}_{-i} \right) \right\}_{i \in \mathcal{N}} \right). 
\end{align}
Namely, the mapping $F$ is obtained by stacking together the subdifferentials of the agents' objective functions with respect to their local decision variables.

Under Assumptions \ref{ass:CCF}$-$\ref{ass:CS}, it follows by \cite[Proposition 12.4]{palomar2010convex} that any solution of GVI$(\bXcal,F)$ is a Nash equilibrium of the game in \eqref{eq:Game}. The inverse implication is not true in general, and actually in passing from the Nash equilibrium problem to the GVI problem most solutions are lost \cite[\S 12.2.2]{palomar2010convex};  indeed, a game may have a Nash equilibrium while the corresponding GVI has no solution. 
Note that, since the cost functions are differentiable (by Assumption \ref{ass:CS}), then GVI$(\bXcal,F)$ reduces to VI$(\bXcal,F)$, which is commonly addressed in the context of game theory via projected gradient algorithms \cite{koshal:nedic:shanbhag:16,paccagnan:16,belgioioso2017convexity}, \cite[\S 12]{facchinei:pang}.

Under the postulated standing assumptions, it is shown in \cite[Proposition~12.11]{palomar2010convex} that a sufficient condition for the existence (and uniqueness) of a variational GNE of the game in \eqref{eq:Game} is the (strict) monotonicity of the pseudo-gradient $F$ in \eqref{eq:PsGr}. 
Thus, let us assume strongly monotonicity of $F$.
\begin{standing}[Strong monotonicity] \label{ass:mon}
The pseudo-gradient $F$ in \eqref{eq:PsGr} is $\eta$-strongly monotone and $\ell_{\text{F}}$-Lipschitz continuous, for some constants $\eta, \, \ell_{\text{F}} > 0$.
 {\hfill $\square$}
\end{standing}
%-------------------------------------------------------

\smallskip
\section{Generalized Nash equilibrium as zero of the sum of two monotone operators} \label{sec:MonRef}

In this section, we exploit operator theory to recast the Nash equilibrium seeking problem into a monotone inclusion, namely, the problem of finding a zero of a set-valued monotone operator.
As first step, we characterize a GNE of the game in terms of KKT conditions of the coupled optimization problems in \eqref{eq:Game}. For each agent $i \in \mathcal{N}$, let us introduce the Lagrangian function $L_i$, defined as 
\begin{equation*}
L_i(\boldx,\lambda_i) := J_i(x_i,\boldx_{-i})+ \iota_{\Omega_i}(x_i) +\lambda_i^\top (A \boldx-b),
\end{equation*}
where $\lambda_i \in \R^m_{\geq 0}$ is the Lagrangian multiplier associated with the coupling constraints.
It follows from \cite[\S 12.2.3]{palomar2010convex} that the set of strategies $\boldx^*$ is a GNE of the game in \eqref{eq:Game} if and only if the following coupled KKT conditions are satisfied:
\begin{equation} \label{eq:KKT}
\begin{cases}
0 \in \partial_{x_i} J_i(x_i^*,\boldx^*_{-i}) + \nc_{\Omega_i}({x}^*_i) + A_i^\top \lambda_i\\
0 \leq \lambda_i \perp -(A {\boldx}^*-b) \geq 0
\end{cases}  \forall i \in \Ncal
\end{equation}

The constraint qualification in Assumption \ref{ass:CS} is needed at this stage to ensures boundedness of the dual variables $\lambda_i$'s.

In a similar fashion, we characterize a variational GNE in term of KKT conditions by exploiting the Lagrangian duality scheme for the corresponding VI problem, see \cite[\S 3.2]{auslender2000lagrangian}. Specifically, ${\boldx}^*$ is a solution of VI$(\bXcal, F)$ if and only if 
${\boldx}^* \in {\argmin}_{\boldy \in \bXcal} \,(\boldy - {\boldx}^*)^\top F(
{\boldx}^*).$
Then, the associated KKT optimality conditions read as
\begin{align} \label{eq:VI-KKT}
\begin{cases}
0 \in \partial_{x_i} J_i({x}^*_i,{\boldx}^*_{-i}) + \nc_{\Omega_i}({x}^*_i) + A_i^\top \mu, 
\ \forall i \in \Ncal
\\
0 \leq \mu \perp -(A {\boldx}^*-b) \geq 0.
\end{cases}
\end{align} 
  
To cast \eqref{eq:VI-KKT} in compact form, we introduce the set-valued mapping $T: \bXcal \times \R^{m}_{\geq 0} \rightrightarrows \R^{nN} \times \R^{m}$, defined as
\begin{align} \label{eq:T}
T: 
\begin{bmatrix}
\boldx\\
\mu
\end{bmatrix}
\mapsto
\begin{bmatrix}
\nc_{\bOmega}(\boldx)+F(\boldx)  + A^\top \mu \\
\nc_{\R^{m}_{\geq 0}}(\mu) - (A \boldx - b)
\end{bmatrix},
\end{align}

Essentially, the role of the mapping $T$ is that its zeros correspond to the variational generalized Nash equilibria of the game in \eqref{eq:Game}, as formalized in the next statement.

\smallskip
\begin{proposition}[{\rm\cite[Th.~1]{belgioioso:grammatico:17cdc}}]
The collective strategy ${\boldx}^*$ is a variational GNE of the game in \eqref{eq:Game} if and only if there exists $\mu^* \in \R^m_{\geq 0}$ such that $\col({\boldx}^*,{\mu}^*) \in \zer\left(T\right)$. Moreover, if $\col({\boldx}^*,{\mu}^*) \in \zer\left(T\right)$, then ${\boldx}^*$ satisfies the KKT conditions in \eqref{eq:KKT} with Lagrangian multipliers $\lambda_i = {\mu}^*$ for all $i \in \Ncal$.
{\hfill $\square$}
\end{proposition}

To conclude this section, we note that the mapping $T$ can be written as the sum of two operators, defined as
\begin{align}
\label{eq:U}
\mc{A}:
\begin{bmatrix}
\boldx\\
\lambda
\end{bmatrix}
&
\mapsto
\begin{bmatrix}
F(\boldx) \\
b
\end{bmatrix},\\
\label{eq:B}
\mathcal{B}:
\begin{bmatrix}
\boldx\\
\lambda
\end{bmatrix}
&
\mapsto
\begin{bmatrix}
\nc_{\bOmega}(\boldx) \\
\nc_{\R^{m}_{\geq 0}}(\mu)
\end{bmatrix} +
\begin{bmatrix}
0 & A^\top \\
-A & 0
\end{bmatrix}
\begin{bmatrix}
\boldx \\
\lambda
\end{bmatrix}.
\end{align}
The formulation $\mc{T} = \mc{A} + \mathcal{B}$ is called \textit{splitting} of $\mc{T}$, and will be exploited in different ways later on. We show next that the mappings $\mc{A}$ and $\mc{B}$ are both monotone, which paves the way for splitting algorithms.

\begin{lemma} \label{lem:U-Bmon}
The mapping $\mathcal{B}$ in \eqref{eq:B} is maximally monotone and $\mc{A}$ in \eqref{eq:U} is $(\eta/\ell_{\text{F}}^2)$-cocoercive.
{\hfill $\square$}
\end{lemma}
\begin{proof}
First, consider $\mathcal{B}=\mathcal{B}_1+\mathcal{B}_2$ in \eqref{eq:B}. The first term $\mathcal{B}_1$ is maximally monotone, since normal cones of closed convex sets are maximally monotone and the concatenation preserves maximality \cite[Prop.~20.23]{bauschke2011convex}; the second term $\mathcal{B}_2$ is linear and skew symmetric, i.e., $\mathcal{B}_2^\top = - \mathcal{B}_2$, thus maximally monotone \cite[Ex.~20.30]{bauschke2011convex}. Then, the maximal monotonicity of $\mathcal{B}$ follows from \cite[Cor.~24.4]{bauschke2011convex}, since $\dom \mathcal{B}_2 = \R^{nN+m}$. 
To prove that $\mc{A}$ is $(\eta /\ell_{\text{F}}^2)$-cocoercive, we note that for all
$\bomega_1=\col(\boldx_1,\lambda_2), \bomega_2= \col(\boldx_2,\lambda_2)\in \R^{nN+m}$, it holds that
\begin{multline*} \textstyle 
 \langle \mc{A}(\bomega_1)- \mc{A}(\bomega), \bomega_1- \bomega_2 \rangle
= \langle F(\boldx_1)- F(\boldx_2), \boldx_1- \boldx_2 \rangle \\
\textstyle
\geq \eta \norm{\boldx_1-\boldx_2}^2  \geq \frac{\eta}{\ell_{\text{F}}^2} \norm{F(\boldx_1) -F(\boldx_2)}^2 \\
\textstyle
 = \frac{\eta}{\ell_{\text{F}}^2} \norm{\mc{A}(\bomega_1)- \mc{A}(\bomega_2)}^2.
\end{multline*}
The first and second inequalities follow from the $\eta$-strong monotonicity and $\ell_{\text{F}}$-Lipschitz continuity, respectively, of the mapping $F$, postulated in Assumption \ref{ass:mon}.
\end{proof}

%------------------------------------------------------
%\newpage
\section{Preconditioned Forward-Backward splitting} \label{sec.IFBS}
In light of Lemma \ref{lem:U-Bmon}, the forward-backward (FB) splitting \cite[\S 25.3]{bauschke2011convex} guarantees convergence to a zero of $\mc{A}+\mc{B}$. In this section, we discuss a design procedure for FB algorithms, which is particularly useful when the resolvent of the operator $\mathcal{B}$ cannot be computed explicitly. Moreover, we show that two existing algorithms for GNE seeking in aggregative games belong to this class of algorithms.

\subsection{Preconditioned Forward-Backward: Design Procedure} \label{sec:IFB}
The main idea of the FB splitting is that the zeros of the mapping $T$ in \eqref{eq:T} correspond to the fixed points of a certain operator which depends on the chosen splitting \eqref{eq:B}$-$\eqref{eq:U}, as formalized next.

\smallskip
\begin{lemma}
For any matrix $\Phi \succ 0$, the following equivalence holds:
\begin{align}
\bomega \in \zer (\mc{A}+\mathcal{B}) \Leftrightarrow \bomega \in \fix \, ( \mathcal{V}_\Phi \circ  \mathcal{U}_\Phi), \label{eq:FB}
\end{align}
where $\mathcal{U}_\Phi := (\Id-\Phi^{-1}\mc{A})$ and $\mathcal{V}_\Phi := (\Id + \Phi^{-1}\mathcal{B})^{-1}$.
{\hfill $\square$}
\end{lemma}
\begin{proof} Consider a vector $\bomega \in \R^{nN+m}$, then
\begin{align*} \textstyle \nonumber
\mathbf{0} \in (\mc{A}+\mathcal{B})(\boldsymbol{\omega})
& \Leftrightarrow \mathbf{0} \in \Phi^{-1}(\mc{A}+\mathcal{B})(\boldsymbol{\omega})\\
\nonumber \textstyle
& \Leftrightarrow (\Id-\Phi^{-1}\mc{A})(\boldsymbol{\omega}) \in 
(\Id + \Phi^{-1}\mathcal{B})(\boldsymbol{\omega})\\
\textstyle
& \Leftrightarrow \boldsymbol{\omega} = \mathcal{V}_\Phi \circ \mathcal{U}_\Phi \; (\boldsymbol{\omega}),
\end{align*}
where the first equivalence holds since $\Phi^{-1} \succ 0$.
\end{proof}
\smallskip

The FB algorithm is the Banach--Picard iteration \cite[(1.67)]{bauschke2011convex} applied to the mappings $\mc{V}_\Phi \circ \mc{U}_{\Phi}$  in \eqref{eq:FB}, i.e., 
\begin{align} \label{eq:P-B}
\boldsymbol{\omega}^{k+1} = (\Id + \Phi^{-1}\mathcal{B})^{-1} \circ (\Id-\Phi^{-1}\mc{A})(\boldsymbol{\omega}^k).
\end{align}
In numerical analysis, $\mathcal{U}_\Phi$ represents a \textit{forward} step with size and direction defined by $\Phi$, while $\mathcal{V}_\Phi$ represents a \textit{backward} step.
Directly from the iteration in \eqref{eq:P-B}, we have that 
\begin{align}
(\Id-\Phi^{-1}\mc{A})(\boldsymbol{\omega}^k) &\in (\Id + \Phi^{-1}\mathcal{B}) (\boldsymbol{\omega}^{k+1})
\Leftrightarrow \nonumber \\[.1cm]
\label{eq:DesEq}
- \mc{A} (\boldsymbol{\omega}^k) &\in \mathcal{B} (\boldsymbol{\omega}^{k+1}) + \Phi (\boldsymbol{\omega}^{k+1}-\boldsymbol{\omega}^k).
\end{align}

The choice of the preconditioning matrix $\Phi$ in \eqref{eq:DesEq} plays a key role in the algorithm design. Next, we provide some general guidelines to design $\Phi$.

\begin{center}
\begin{minipage}{\columnwidth}
\hrule
\smallskip
\textit{Design guidelines for the preconditioning matrix $\Phi$}: 
\smallskip
\hrule
\medskip
\begin{enumerate}[$1.$]
\item $\forall \xi \in \rm{eig}(\Phi)$,
$ \Re \left[ \xi \right] > 0$ (necessary);
\item $\bomega^{k+1}$ in \eqref{eq:DesEq} explicitly computable (necessary);
\item $\Phi = \Phi^\top$ (convenient convergence analysis);
\item iterations in \eqref{eq:P-B} sequential (convenient implementation).
\end{enumerate}
\hrule
\end{minipage}
\end{center}
Without loss of generality, we denote $\boldsymbol{\omega} = \col(\boldx, \lambda)$, then the inclusion in \eqref{eq:DesEq} reads in expanded form as
\begin{multline} \label{eq:DesignMech} \textstyle
- \begin{bmatrix}
F(\boldx^k)\\ b
\end{bmatrix} \in
\begin{bmatrix}
\nc_{\Omega}(\boldx^{k+1})\\
\nc_{\R^m_{\geq 0}}(\lambda^{k+1})
\end{bmatrix}
+
\begin{bmatrix}
 0 & A^\top\\
-A & 0\\ 
\end{bmatrix}
\begin{bmatrix}
\boldx^{k+1} \\
\lambda^{k+1}
\end{bmatrix}
%-----------
\\ \textstyle
+
\begin{bmatrix}
\Phi_{11} & \Phi_{12}  \\
\Phi_{21} & \Phi_{22} 
\end{bmatrix}
\begin{bmatrix}
\boldx^{k+1}- \boldx^k \\
\lambda^{k+1}-\lambda^k
\end{bmatrix}.
\end{multline}

Consider a symmetric matrix $\Phi_{\text s}$, designed accordingly to the guidelines above, i.e.,
\begin{align} \label{eq:Phi_s} \textstyle
\Phi_{\text s} :&= \begin{bmatrix}
\boldsymbol{\alpha}^{-1} & -A^\top \\
-A & \gamma^{-1}I_m
\end{bmatrix},
\end{align}
where $\boldsymbol{\alpha} := \diag(\alpha_1,\cdots,\alpha_N) \otimes I_n$ and the coefficients (step sizes) $\{ \alpha_i \}_{i=1}^N$ and $\gamma$ are chosen such that $\Phi_{\text s}$ has positive eigenvalues (guideline 1), as formalized in the next lemma.

\smallskip
\begin{lemma} \label{lem:Phi_s}
The matrix $\Phi_{\text s}$ in \eqref{eq:Phi_s} is positive definite if
\begin{align} \textstyle \label{eq:defPos}
\gamma < (\norm{A}^2  \alpha_{i,\max})^{-1}, \quad \gamma, \, \alpha_{i,\min} > 0,
\end{align}
where $\alpha_{i,\max}:= \max_{i \in \Ncal} \alpha_i$ and $\alpha_{i,\min}:= \min_{i \in \Ncal} \alpha_i$. 
\end{lemma}

\smallskip
\begin{proof}
The conditions in \eqref{eq:defPos} directly follow by applying the the Schur's complement on $\Phi_{\text s}$ in \eqref{eq:Phi_s}.
\end{proof}
\smallskip
Now, we present the \textit{preconditioned Forward-Backward} (pFB) algorithm associated with $\Phi_{\text s}$, which is the Banach--Picard iteration in \eqref{eq:P-B} for $\Phi = \Phi_{\text s}$. 
\begin{center}
\begin{minipage}{\columnwidth}
\hrule
\smallskip
\textit{Algorithm 1}: Preconditioned Forward Backward (pFB)
\smallskip
\hrule
\medskip
\noindent
\begin{align} 
\boldx^{k+1} &= \proj_{\bOmega}\,\big[ \boldx^k - \boldsymbol{\alpha} ( F(\boldx^{k}) + A^\top \lambda^k \big ) ] \nonumber \\
\lambda^{k+1} &= \proj_{\R^m_{\geq 0}} \big[\lambda^k + \gamma (2A\boldx^{k+1}-A\boldx^k -b) \big] \nonumber 
\end{align}
\hrule
\end{minipage}
\end{center}

\smallskip
\begin{remark}
The iterations of Algorithm 1 are sequential (guideline 4), namely, the multiplier update, $\lambda^{k+1}$, exploits the most recent value of the agents' strategies, $\boldx^{k+1}$.
{\hfill $\square$}
\end{remark}
\smallskip
In the next statement, we show the convergence of Algorithm 1 to a variational generalized Nash equilibrium, under suitable choices of the step sizes.

\smallskip
\begin{theorem}[Global convergence of pFB] \label{the:APAFB}
The sequence $\left( \col(\boldx^k,\lambda^k) \right)_{k=0}^\infty$ defined by Algorithm 1, with step sizes $\alpha_i \in ( 0, 2 \eta/\ellF^2 )$, for all $i \in \Ncal$, and $\gamma \in (0, \gamma_{\max})$, with $ \textstyle
 \gamma_{\max} := \frac{1}{\norm{A}^2} (\frac{1}{ \alpha_{i,\max} } - \frac{1}{2 \eta/ \ell_{\text{F}}^2})$, globally
converges to some $\col({\boldx}^*,{\lambda}^*) \in \zer  (T)$, with $T$ as in \eqref{eq:T}.
{\hfill $\square$}
\end{theorem}
\smallskip

\begin{proof}
See Section \ref{sec:CA}.
\end{proof}

%---------------------------------------------------
\subsection{The Asymmetric Projection Algorithm in \cite[Alg.~1]{paccagnan:16} is a preconditioned Forward-Backward splitting}
We note that Algorithm 1 (with equal step sizes)  corresponds to the ``asymmetric'' projected algorithm (APA) proposed in \cite[Alg.~1]{paccagnan:16}. Therein, the algorithm design and its convergence analysis rely on a variational inequality formulation of the Nash equilibrium problem. Specifically, the authors define the convex set $\mathcal{C} := \bOmega \times \R^{m}_{\geq 0}$ and the monotone mapping 
\begin{equation*}
\mathcal{R}: \begin{bmatrix}
\boldx \\ \lambda
\end{bmatrix} \mapsto
\begin{bmatrix}
F(\boldx) \\
b
\end{bmatrix}+
\begin{bmatrix}
0 & A^\top \\
-A & 0
\end{bmatrix}
\begin{bmatrix}
\boldx \\
\lambda
\end{bmatrix},
\end{equation*}
and characterize the GNE as solutions of VI$(\mathcal{C},\mathcal{R})$. Then, to solve VI$(\mathcal{C},\mathcal{R})$ in a semi-decentralized fashion, the authors propose an asymmetric implementation of the projection algorithm for variational inequalities \cite[12.5.1]{facchinei:pang}, in which each iteration is computed as
\begin{align} \label{eq:VI-iter}
\bomega^{k+1} = \text{ solution to  VI}(\mathcal{C},\mathcal{R}^k_D),
\end{align}
where $\mathcal{R}^k_D(\bomega) := \mathcal{R}(\bomega^k) + D(\bomega - \bomega^k)$ and
\begin{equation} \label{eq:D}
D := \begin{bmatrix}
\tau^{-1}I &0 \\
-2A        &\tau^{-1I}
\end{bmatrix}.
\end{equation}
If the parameter $\tau>0 $ in $ \eqref{eq:D}$ is chosen such that $D\succ 0$, then the unique solution in \eqref{eq:VI-iter} is $\proj_{\mathcal{C},D} \big( \bomega^k - D^{-1}\mathcal{R}(\bomega^k) \big)$, where $\proj_{\mathcal{C},D}$ is the projection operator characterized by the asymmetric matrix $D$ in \eqref{eq:D}.
Thus, the iteration in \eqref{eq:VI-iter} equivalently reads as
\begin{align} \label{eq:APA2}
\bomega^{k+1} = (\Id + D^{-1} \nc_{\mathcal{C}})^{-1} \circ (\Id - D^{-1} \mathcal{R})(\bomega^k),
\end{align}
which is nothing but a pFB associated with the splitting $T:= \nc_{\mathcal{C}}+\mathcal{R}$ and the preconditioning matrix $\Phi = D$. 

\smallskip
\begin{remark}
With \eqref{eq:APA2}, we showed that pFB algorithms based on different splittings and preconditioning matrices can lead to the same algorithm. From an operator-theoretic perspective, the convergence analysis is more convenient for the pFB with symmetric matrix $\Phi_{\text s} = \frac{1}{2}(D+D^\top)$ (guideline 3), since 
the properties that $\mc{A}$, $\mathcal{B}$ have with the standard inner product $\langle \cdot,\cdot \rangle_{I}$ are preserved for $\Phi_{\text s}^{-1}\mc{A}$, $\Phi_{\text s}^{-1} \mathcal{B}$ with inner product $\langle \cdot,\cdot \rangle_{\Phi_{\text s}} = \langle \Phi \cdot,\cdot \rangle $, as shown in Section \ref{sec:CA}. 
\hfill $\square$
\end{remark}

%--------------------------------------------------------
%\newpage
\subsection{The distributed algorithm for aggregative games on graphs in \cite[\S 3]{koshal:nedic:shanbhag:16} is a preconditioned Forward-Backward}
In this section, we show that the synchronous distributed algorithm for NE seeking in network aggregative games proposed in \cite[\S 3]{koshal:nedic:shanbhag:16} can be written as a pFB splitting.

In \cite{koshal:nedic:shanbhag:16}, the authors consider an aggregative game without coupling constraints, i.e., $\Xcal_i(\cdot) = \Omega_i$ for all $i \in \mathcal{N}$, and wherein the agents have no access to the aggregate decision \eqref{eq:sigma}, but build an estimate of it by communicating over an undirected network with their neighboring agents.

Specifically, let $\mathscr{E}$ be the set of underlying undirected edges between agents; let $\Ncal_i := \{j \in \Ncal \,| \; (i,j) \in \mathscr{E} \}$ denote the set of neighbors of agent $i$, with the convention that $i \not\in \Ncal_i$; let $D:=\diag(d_1, \cdots, d_N) $ be the degree matrix, where $d_i := |\Ncal_{i}|$; let $E$ be the adjacency matrix, such that
$[E]_{ij} := 1$ if $j \in \Ncal_i$, $0$ otherwise.
%
%\begin{align*}
%[E]_{ij} :=
%\begin{cases}
%1 & \text{if } j \in \Ncal_i \\
%0 & \text{otherwise;}
%\end{cases}
%\end{align*}
%
let $L := D-E$ be the Laplacian matrix and let us define the matrix $W := \mathrm{J}_D(I+E)$, with $\mathrm{J}_D = (I+D)^{-1}$, such that, given a vector $\boldv = \col(v_1,\cdots,v_N)$, with $v_i \in \R^n$,  then
\begin{align*} \textstyle
[(W \otimes I_n) \boldv ]_i = \frac{1}{|\Ncal_i|+1}\sum_{j \in \Ncal_i \cup \{ i \}} v_j, \quad \text{for all } i \in \Ncal.
\end{align*}
Let $F_{\sigma}:\Xcal \times \R^{nN} \rightrightarrows \R^{nN}$ be the extension of $F$ in \eqref{eq:PsGr} to the augmented space of actions and estimates, defined as
\begin{align} 
 \label{eq:ExPsGr}
F_{\sigma}(\boldx,\boldz) &:= 
\col\left( \left\{  \partial_{x^i} \, f_i \left( x_i, \, z_i \right) \right\}_{i \in \mathcal{N}} \right) 
\end{align}
Note that $F_{\sigma}\left(\boldx, \mathbf{1}_N \otimes (M\boldx) \right)  = F(\boldx)$.
Next, we present a static version of the algorithm in \cite[\S 3]{koshal:nedic:shanbhag:16}, whose convergence to a NE is established in \cite[Prop. 2]{koshal:nedic:shanbhag:16}.
\smallskip 

\begin{center}
\begin{minipage}{\columnwidth}
\hrule
\smallskip
\textit{Algorithm 2}: Koshal--Nedi\'c--Shanbhag algorithm
\smallskip
\hrule
\begin{align*}
\boldx^{k+1} &= \proj_{\bXcal} \left[ 
\boldx^k - \alpha F_{\sigma} \big( \boldx, (W \otimes I_n) \boldv^k \big) 
\right] \\
\boldv^{k+1} &= \proj_{\R^{nN}} \left[  (W \otimes I_n) \boldv^k +\boldx^{k+1}-\boldx^k \right]
% \\
%&=\mathrm{J}_{D\otimes I_n} \left[ \boldv^k +(E \otimes I_n)\boldv^k + \right. \\
%& \left. \hspace*{9em}
%(I+E \otimes I_n)(\boldx^{k+1}-\boldx^k) \right]
\end{align*}
\hrule
\smallskip
\end{minipage}
\end{center}

In the following statement, we show that Algorithm 2 is a pFB splitting with symmetric preconditioning matrix.

\smallskip
\begin{proposition}
Let the mappings $\mc{A}$, $\mathcal{B}$ and the preconditioning matrix $\Phi$ be defined as
\begin{align} \label{eq:Uns} \textstyle
%\label{eq:Uns}
\mc{A}&:
\begin{bmatrix}
\boldx\\
\bolds
\end{bmatrix}
\rightarrow
\begin{bmatrix}
F_{\sigma}(\boldx,\bolds) \\
0
\end{bmatrix} + \tfrac{1}{2}
\begin{bmatrix}
0       & -P \\
P & 0 
\end{bmatrix}
\begin{bmatrix}
\boldx\\
\bolds
\end{bmatrix}, \\
\label{eq:Bns}  \textstyle
\mathcal{B} &:
\begin{bmatrix}
\boldx\\
 \bolds
\end{bmatrix}
\rightarrow
\begin{bmatrix}
\nc_{\bOmega}(\boldx) \\
L_n \bolds
\end{bmatrix} + \tfrac{1}{2}
\begin{bmatrix}
0 & P \\
-P & 0 
\end{bmatrix}
\begin{bmatrix}
\boldx\\
\bolds
\end{bmatrix}, \\
\label{eq:Phins} \textstyle
\Phi &:= \begin{bmatrix}
\alpha^{-1}I & -\frac{1}{2}P  \\
-\frac{1}{2}P       & \hspace*{1.2em}P
\end{bmatrix},
\end{align}
where $L_n := L \otimes I_n$ and $P := (I + E \otimes I_n)$. Then, the sequence ${\left( \col (\boldx^k,\bolds^k) \right)}_{k=0}^\infty$ generated by the pFB in \eqref{eq:P-B}, with $\mc{A}$, $\mathcal{B}$, $\Phi$ as in \eqref{eq:Uns}$-$\eqref{eq:Phins}, corresponds to the sequence ${\left( \col (\boldx^k,(W \otimes I_n) \boldv^k) \right)}_{k=0}^\infty$ generated by Algorithm 2. 
\hfill $\square$
\end{proposition}

\smallskip
\begin{proof} The iteration in Alg. 2 can be derived by solving the inclusion \eqref{eq:DesEq} with $\mc{A}$, $\mathcal{B}$, $\Phi$ as in \eqref{eq:Uns}$-$\eqref{eq:Phins}, for $\boldx^{k+1}$ and $\bolds^{k+1}$ and noticing that $\bolds^k = (W \otimes I_n) \boldv^k$. 
\end{proof}
%

%-------------------------------------------------------
%\newpage
\section{Convergence Analysis} \label{sec:CA}

First, we show that the properties the mappings $\mc{A}$, $\mathcal{B}$ have with the standard inner product are preserved for $\Phi_{\text s}^{-1}\mc{A}$, $\Phi_{\text s}^{-1} \mathcal{B}$ with the inner product $\langle \cdot,\cdot \rangle_{\Phi_{\text s}} = \langle \Phi \cdot,\cdot \rangle$.

\smallskip
\begin{lemma} \label{lem:mapsReg}
The mappings $\Phi_{\text s}^{-1} \mc{A}$, $\Phi_{\text s}^{-1} \mathcal{B}$, $\mc{U}_{\Phi_\text{s}}$, $\mc{V}_{\Phi_\text{s}}$ satisfy the following properties in the $\Phi_{\text s}$-induced norm:
\begin{enumerate}[(i)]
\item $\Phi_{\text s}^{-1} \mc{A}$ is $\beta$-cocoercive
and $\mc{U}_{\Phi_\text{s}}$ is $\frac{1}{2 \beta}$-averaged, where \\
$\beta := \lambda_{\min}(\boldsymbol{\alpha}^{-1} - \gamma A^\top A)\frac{\eta}{\ell_{F}}$. 
\item $\Phi_{\text s}^{-1} \mathcal{B}$ is maximally monotone and $\mc{V}_{\Phi_\text{s}}$ is $\frac{1}{2}$-averaged.
\end{enumerate}
\end{lemma}
\smallskip

\begin{proof}
(i): We need to show that for all $\bomega_1, \bomega_2 \in \bOmega \times \R^m_{\geq 0}$ the following condition holds:
\begin{multline} \label{eq:Ucoco} \textstyle
\langle \Phi_{\text s}^{-1} \mc{A}(\bomega_1)-\Phi_{\text s}^{-1}\mc{A}(\bomega_2), \; \bomega_1-\bomega_2 \rangle_{\Phi_{\text s}} \\
\textstyle
\geq \beta \norm{\Phi_{\text s}^{-1} \mc{A}(\bomega_1)-\Phi_{\text s}^{-1}\mc{A}(\bomega_2)}^2_{\Phi_{\text s}}.
\end{multline}
We first provide an upper bound for the right hand side of \eqref{eq:Ucoco}. Let us denote $\bomega_i = \col(\boldx_i,\mu_i)$ for $i=1,2$, then
\begin{align} \textstyle \nonumber
&\norm{\Phi_{\text s}^{-1} \mc{A}(\bomega_1)-\Phi_{\text s}^{-1}\mc{A}(\bomega_2)}^2_{\Phi_{\text s}} \\
\textstyle \nonumber
&= \langle\ \Phi_{\text s} \Phi_{\text s}^{-1} ( \mc{A} \big(\bomega_1)-\mc{A}(\bomega_2) \big), \ \Phi_{\text s}^{-1} ( \mc{A} \big(\bomega_1)-\mc{A}(\bomega_2) \big) \, \rangle = \\
\nonumber \scriptstyle
& \scriptstyle  \begin{bmatrix}
F(\boldx_1)-F(\boldx_2) \\ 0
\end{bmatrix}^\top
\begin{bmatrix}
[\Phi_{\text s}^{-1}]_{11} &[\Phi_{\text s}^{-1}]_{12}  \\
[\Phi_{\text s}^{-1}]_{21} &[\Phi_{\text s}^{-1}]_{22} 
\end{bmatrix}
\begin{bmatrix}
F(\boldx_1)-F(\boldx_2) \\ 0
\end{bmatrix} \\
\textstyle \nonumber
&= \norm{ F(\boldx_1)-F(\boldx_2)}_{[\Phi_{\text s}^{-1}]_{11}} 
\\
& \leq \norm{[\Phi_{\text s}^{-1}]_{11}}_2 \norm{F(\boldx_1)-F(\boldx_2)}^2   \textstyle      \nonumber \\
\textstyle \label{eq:UBo}
&= \norm{[\Phi_{\text s}^{-1}]_{11}}_2 \norm{\mc{A}(\bomega_1)-\mc{A}(\bomega_2)}^2,
\end{align}
where $[\Phi_{\text s}^{-1}]_{11}:= (\boldsymbol{\alpha}^{-1} \!-\gamma A^\top A)^{-1}$ is symmetric and positive definite if the step sizes $\alpha_i$, $\gamma$ are chosen as in Lemma \ref{lem:Phi_s}. Moreover, it holds that $\norm{[\Phi_{\text s}^{-1}]_{11}}_2 = 1/ \lambda_{\min}([\Phi_{\text s}^{-1}]_{11}^{-1}) $, where $\lambda_{\min}([\Phi_{\text s}^{-1}]_{11}^{-1})$ is the smallest eigenvalue of $[\Phi_{\text s}^{-1}]_{11}^{-1}$.
Now, we exploit the $\frac{\eta}{\ell_{F}^2}$-cocoercivity of $\mc{A}$ and the upper bound in \eqref{eq:UBo} to define the cocoercivity constant $\beta$ in \eqref{eq:Ucoco}.
\begin{align*}
\textstyle
& \langle \Phi_{\text s}^{-1} \mc{A}(\bomega_1)-\Phi_{\text s}^{-1}\mc{A}(\bomega_2), \; \bomega_1-\bomega_2 \rangle _{\Phi_{\text s}} \\
\textstyle
& = \langle \mc{A}(\bomega_1)-\mc{A}(\bomega_2), \bomega_1-\bomega_2 \rangle\
\geq \frac{\eta}{\ell_{F}^2} \norm{\mc{A}(\bomega_1)-\mc{A}(\bomega_2)}^2 \\
\textstyle
%& \geq  \frac{1}{\norm{[\Phi_{\text s}^{-1}]_{11}^{-1}}_2} \frac{\eta}{\ell_{F}^2} \norm{\Phi_{\text s}^{-1} \mc{A}(\bomega_1)-\Phi_{\text s}^{-1}\mc{A}(\bomega_2)}^2_{\Phi_{\text s}} \\
%%
\textstyle
&\geq \frac{\eta}{\ell_{F}^2} \lambda_{\min}([\Phi_{\text s}^{-1}]_{11}^{-1})  \norm{\Phi_{\text s}^{-1} \mc{A}(\bomega_1)-\Phi_{\text s}^{-1}\mc{A}(\bomega_2)}^2_{\Phi_{\text s}}.%
\end{align*}
Thus, the mapping $\Phi_{\text s}^{-1} \mc{A}$ is cocoercive with constant $\beta := \frac{\eta}{\ell_{F}^2} \lambda_{\min}([\Phi_{\text s}^{-1}]_{11}^{-1}) $ w.r.t. the $\Phi_{\text s}$-induced norm.
Since $\Phi_{\text s}^{-1}\mc{A}$ is $\beta$-cocoercive, it follows from \cite[Prop.~4.33]{bauschke2011convex} that $\mc{U}_{\Phi_\text{s}} = (\Id-\Phi_{\text s}^{-1}\mc{A})$ is $\frac{1}{2 \beta}$-averaged.
(ii): Since $\mathcal{B}$ is maximally monotone by Lemma \eqref{lem:U-Bmon} and $\Phi_s^{-1}$ is positive definite, if the step sizes are chosen as in Lemma \ref{lem:Phi_s}, then the maximal monotonicity of $\Phi_s^{-1} \mathcal{B}$ follows from \cite[Lemma~3.7]{combettes2014variable}.
\end{proof}
\smallskip

Next, we show that the FB operator $\mc{V}_{\Phi_\text{s}} \circ \mc{U}_{\Phi_\text{s}}$ is averaged if the step sizes are chosen small enough.

\smallskip
\begin{lemma} \label{lem:UBs}
The FB operator $\mc{V}_{\Phi_\text{s}} \circ \mc{U}_{\Phi_\text{s}}$ in \eqref{eq:FB}, with $\Phi = \Phi_{\text s}$, is $\theta$-averaged, with $\theta := \frac{1}{2 - 1/(2\beta)} \in (0,1)$, if
\begin{align} \textstyle \label{eq:SSC}
\gamma < \frac{1}{\norm{A}^2}\left( \frac{1}{\alpha_{i,\max} } - \frac{1}{2 \eta/ \ell_{\text{F}}^2  }
\right), \quad \alpha_{i,\max} < \frac{2\eta}{\ellF^2}.
\end{align}
Moreover, if $\alpha_i = \gamma$ for all $ i \in \Ncal$, then \eqref{eq:SSC} reads as
\begin{align} \textstyle \label{eq:SSCe}
\gamma < \frac{-1+\sqrt{1+ \norm{A}^2 \left( 4 \eta/ \ell_F^2 \right)^2  }}{\norm{A}^2  (4 \eta/\ellF^2) }.
\end{align}
\end{lemma}
\smallskip

\begin{proof}
By Lemma \ref{lem:mapsReg}, $\mc{U}_{\Phi_\text{s}}$ and $\mc{V}_{\Phi_\text{s}}$ are averaged with constants $\tau_1:=\frac{1}{2\beta}$ and $\tau_2:= \frac{1}{2}$, respectively. 
If $\tau_1 \in (0,1)$, then $\mc{V}_{\Phi_\text{s}} \circ \mc{U}_{\Phi_\text{s}}$ is $\theta$-averaged with $\theta = \frac{\tau_1+\tau_2-2\tau_1 \tau_2}{1-\tau_1\tau_2}=\frac{2\beta}{4\beta - 1} \in (0,1)$, by \cite[proposition~2.4]{combettes:yamada:15}.
To conclude, we note that the following condition implies that $\tau_1 < 1$:
\begin{multline} \textstyle
\beta = \frac{\eta}{\ellF^2} \lambda_{\min}[(\boldsymbol{\alpha}^{-1} -\gamma A^\top A)^{-1}]  \\
\textstyle
\geq
\frac{\eta}{\ellF^2} (\frac{1}{\alpha_{i,\max}}-\gamma \norm{A}^2)  >
 \frac{1}{2}, \label{eq:S-In}
\end{multline}
where the second inequality in \eqref{eq:S-In} holds for step sizes chosen as in \eqref{eq:SSC}. Moreover, if the step sizes are equal, i.e., $\alpha_i = \gamma$ for all $ i \in \Ncal$, then $\eqref{eq:S-In}$ holds for $\gamma$ as in \eqref{eq:SSCe}.
\end{proof}
%
%
%\begin{figure}
%\centering
%\include{Figures/boundsL}
%\caption{ \footnotesize The dashed lines represent the upper bounds of the step size $\gamma$ in function of the cocoercivity parameter $\eta/\ell_F^2$ of the pseudo gradient $F$
%%in \eqref{eq:PsGr}
%, for $\norm{A}^2=1$. The correspondent shaded areas represent the feasible choices for the step sizes. }
%\label{fig:UP}
%\end{figure}
%
\smallskip
We can now prove the convergence of Algorithm 1.

\smallskip
\textit{Proof of Theorem \ref{the:APAFB}:}
The iterations in Alg. 1 are obtained explicitly by substituting $\Phi_{\text s}$ into \eqref{eq:DesignMech} and solving for $\boldx^{k+1}$, $\lambda^{k+1}$. Thus, Alg. 1 is the Banach--Picard iteration of the mapping $\mc{V}_{\Phi_\text{s}} \circ \mc{U}_{\Phi_\text{s}}$, which is $\theta$-averaged, with $\theta \in (0,1)$, by Lemma \ref{lem:UBs}, if the step sizes satisfy \eqref{eq:SSC}.
The convergence of the sequence ${\left( \col(\boldx^k,\lambda^k) \right)_{k=0}^\infty}$ generated by the Banach--Picard iteration of $\mc{V}_{\Phi_\text{s}} \circ \mc{U}_{\Phi_\text{s}}$ to $\col(\bar{\boldx},\bar{\lambda}) \in \fix(\mc{V}_{\Phi_\text{s}} \circ \mc{U}_{\Phi_\text{s}}) = \zer (\mc{A}+ \mathcal{B}) \neq \emptyset$ follows by \cite[Prop.~15.5]{bauschke2011convex}.
\hfill $\blacksquare$

\smallskip
\begin{remark}
The upper bounds in Lemma \ref{lem:UBs} are increasing functions of the cocoercivity constant $\eta/\ellF^2$ of $F$ in \eqref{eq:PsGr}. In particular, the upper bound for the case with equal step sizes in \eqref{eq:SSCe} is tighter than that obtained in \cite[Theorem~2]{paccagnan:16}.
{\hfill $\square$} 
\end{remark}

%-------------------------------------------------------
\section{Conclusion} \label{sec:Concl}
By monotone operator theory, projected-gradient methods for generalized Nash equilibrium seeking in aggregative games are preconditioned forward-backward splitting methods, whose convergence has been established for problems with strongly monotone pseudo-gradient mapping.

\bibliographystyle{IEEEtran}
\bibliography{library}

\end{document}